\documentclass[12pt]{amsart}
\usepackage[utf8]{inputenc}
\input{commands.tex}

\renewcommand{\P}{\mathbb{P}}
\let\til\widetilde
\providecommand{\id}{\text{id}}
\providecommand{\GW}{\text{GW}}
\providecommand{\Spec}{\text{Spec}\hspace{0.1em}}
\providecommand{\Tr}{\text{Tr}}
\providecommand{\A}{\mathbb{A}}
\providecommand{\G}{\mathbb{G}}

\let\minus\smallsetminus
\let\bar\overline
\let\smashprod\wedge
\providecommand{\SH}{\mathcal{SH}}
\newcommand{\mf}[1]{\mathfrak{#1}}

\usepackage{hyperref}
\usepackage[margin=2.9cm]{geometry}

\renewcommand{\O}{\mathcal{O}}
\newcommand{\Th}{\text{Th}}

\providecommand{\ind}{\text{ind}}
\providecommand{\Hom}{\text{Hom}}
\providecommand{\End}{\text{End}}
\providecommand{\Spc}{\text{Spc}}

\title{The trace of the local \texorpdfstring{$\mathbb{A}^1$}{A1}-degree}
\author[Brazelton]{Thomas Brazelton}
\author[Burklund]{Robert Burklund}
\author[McKean]{Stephen McKean}
\author[Montoro]{Michael Montoro}
\author[Opie]{Morgan Opie}
\date{\today}
\begin{document}
\begin{abstract}
We prove that the local $\A^1$-degree of a polynomial function at an isolated zero with finite separable residue field is given by the trace of the local $\A^1$-degree over the residue field. This fact was originally suggested by Morel's work on motivic transfers, and by Kass and Wickelgren's work on the Scheja--Storch bilinear form. As a corollary, we generalize a result of Kass and Wickelgren relating the Scheja--Storch form and the local $\A^1$-degree.
\end{abstract}
\maketitle

\section{Introduction}
The $\A^1$-degree, first defined by Morel \cite{Morel-Trieste,Morel-book}, provides a foundational tool for solving problems in $\A^1$\textit{-enumerative geometry}.\footnote{$\A^1$-enumerative geometry is the application of $\A^1$-homotopy theory to the study of enumerative geometry over arbitrary fields. For details, see the expository paper \cite{unstablehandbook}, as well as the exposition found in \cite{KW17,Levine,Bethea,SW,KW-EKL,VogtLarson}.} In contrast to classical notions of degree, the local $\mathbb{A}^1$-degree is not integer valued: given a polynomial function $f:\A^n_k \to \A^n_k$ with isolated zero $p$, the local $\A^1$-degree of $f$ at $p$, denoted by $\deg_p^{\A^1}(f)$, is defined to be an element of $\GW(k)$, the Grothendieck--Witt group of the ground field. 
\begin{definition}
Let $k$ be a field. The \textit{Grothendieck--Witt group} $\GW(k)$ is defined to be the group completion of the monoid of isomorphism classes of symmetric non-degenerate bilinear forms over $k$. The group operation is the direct sum of bilinear forms. We may also give $\GW(k)$ a ring structure by taking tensor products of bilinear forms for our multiplication.
\end{definition}
At rational points, the local $\A^1$-degree can be related to other important invariants. The Scheja--Storch form is another $\GW(k)$-valued invariant defined via a duality on the local complete intersection cut out by the components of a given polynomial map (see  Subsection~\ref{subsec:SS} for details). Kass and Wickelgren show that the isomorphism class of the Scheja--Storch bilinear form \cite{SchejaStorch} is equal to the local $\A^1$-degree at rational points \cite{KW-EKL}. Kass and Wickelgren also show that at points with finite separable residue field, the Scheja--Storch form is given by taking the trace of the Scheja--Storch form over the residue field \cite[Proposition~32]{KW17}. 

In practice, one may need to consider the local $\A^1$-degree at non-rational points. This is the case of interest to us. At points whose residue field is a finite extension of the ground field, Morel's work on cohomological transfer maps \cite{Morel-book} suggests the following formula: the local $\A^1$-degree at a non-rational point should be computed by first taking the local $\A^1$-degree over the residue field, and then by post-composing with a field trace. This suggestion is supported by the aforementioned results of Kass--Wickelgren on the Scheja--Storch form \cite[Proposition~32]{KW17}. Our main result is to confirm this formula. We state our result precisely in Theorem~\ref{thm:EKL-is-deg}, after introducing necessary terminology. 

\begin{definition}\label{def:trace-GW}
Given a separable field extension $L/k$ of finite degree, the {\it trace} $$\Tr_{L/k}:\GW(L)\to\GW(k)$$ is given by post-composing the field trace (which we also denote $\Tr_{L/k}$). That is, if $\beta:V\times V\to L$ is a representative of an isomorphism class of symmetric non-degenerate bilinear forms over $L$, then its trace is the isomorphism class of the following symmetric bilinear form over $k$
\[\Tr_{L/k}\beta:V\times V\to L\xrightarrow{\Tr_{L/k}}k.\]
\end{definition}

When $p$ is not a $k$-rational point, we can lift $f$ to a function $f\otimes_k k(p):\A^n_{k(p)}\to\A^n_{k(p)}$ after fixing a choice of field embedding $k \hookto k(p)$. Moreover, we may lift $p$ to an isolated $k(p)$-rational zero $\til{p}$ of $f\otimes_k k(p)$, and we thus obtain the local degree $\deg_{\til{p}}^{\A^1}(f\otimes_k k(p))\in\GW(k(p))$. 

We can now state our main result:

\begin{theorem}\label{thm:EKL-is-deg} Let $k$ be a field, $f:\A_k^n \to \A_k^n$ be an endomorphism of affine space, and let $p\in \A_k^n$ be an isolated root of $f$ such that $k(p)$ is a separable extension of finite degree over $k$. Let $\til{p}$ denote the canonical point above $p$. Then
\begin{align*}
    \deg^{\A^1}_p f = \Tr_{k(p)/k} \deg_{\til{p}}^{\A^1} \left(f \otimes_k k(p) \right)
\end{align*}
 in $\GW(k)$.
\end{theorem}

As a corollary, we strengthen Kass and Wickelgren's result relating the local $\A^1$-degree and the Scheja--Storch form \cite{KW-EKL} by weakening the requirement that the point be rational.

\begin{corollary}\label{cor:SS-is-local-deg} At points whose residue fields are finite separable extensions of the ground field, the local $\A^1$-degree coincides with the isomorphism class of the Scheja--Storch form.
\end{corollary}

In this paper we utilize the machinery of stable $\A^1$-homotopy theory, initially developed by Morel and Voevodsky \cite{MV}, as well as the six functors formalism in this setting \cite{Ayoub,Cisinski}. We also rely heavily on results of Hoyois \cite{hoyois} to prove our main result. After working in the stable $\A^1$-homotopy category, we apply Morel's $\A^1$-degree to obtain the desired equality in $\GW(k)$.

\begin{conventions}\label{conventions}
Throughout, we adopt the following conventions: 
\begin{itemize}
\item We will use $k$ to denote a general field. If $p$ is a point of a $k$-scheme, with $k(p)/k$ a separable extension of finite degree, we call $p$ a \textit{finite separable point}. We may also say that $p$ has a \textit{finite separable residue field} in this context. We remark that all such points are closed points. 

\item Whenever a finite separable point $p$ is chosen, we also choose an embedding $k \to k(p),$ to remain fixed for all subsequent discussion of $p$. The associated structure map we will always denote by $\rho : \Spec k(p) \to \Spec k$.
\item Given a scheme $X$ over $k$, we denote the base change $X\times_{\Spec{k}}\Spec{k(p)}$ by $X_{k(p)}$.
\item The structure map $\rho$ allows us to define the canonical $k(p)$-rational point in $\A^n_{k(p)}$ sitting above $p$, which we denote by $\tilde p$. 
\end{itemize}
\end{conventions}

\subsection{Acknowledgements}
We would like to thank Kirsten Wickelgren for her excellent guidance during this project. We would also like to thank Matthias Wendt and Jesse Kass for their helpful insights during the 2019 Arizona Winter School. We are grateful to an anonymous referee for comments and suggestions that greatly improved the paper. Finally, we would like to thank the organizers of the Arizona Winter School, as this work is the result of one of the annual AWS project groups. 

The first and fifth authors are supported by NSF Graduate Research Fellowships under grant numbers DGE-1845298 and DGE-1144152, respectively.

\section{Preliminaries}

In this section, we introduce the main notions necessary to state and prove Theorem~\ref{thm:EKL-is-deg}. We begin in Subsection~\ref{subsec:local-a1-deg} by defining the local $\A^1$-degree. In Subsection~\ref{subsec:SHK}, we highlight key properties of the stable motivic homotopy category in the form that we will need them. Finally, in Subsection~\ref{subsec:SS}, we discuss the Scheja--Storch form. 

We will assume some familiarity with motivic homotopy theory. For more detail about the category of motivic spaces $\Spc_k^{\A^1}$ and the unstable motivic homotopy category $\mathcal{H}(k)$, we refer the reader to the excellent expository articles \cite{antieau2016primer,unstablehandbook}. For the construction of the stable motivic homotopy category $\SH(k)$, we refer the reader to \cite{Morel-Trieste}. 

\begin{notation}\label{notation:stable-motivic-homotopy}
We denote by $\SH(k)$ the stable motivic homotopy category over the scheme $\Spec k$. The sphere spectrum in this category will be denoted by $\mathbf{1}_k$. We will also use $\left[ -,- \right]_{\A^1}$ to denote $\A^1$-weak equivalence classes of maps between two motivic spaces, by which we mean a hom-set in the homotopy category $\mathcal{H}(k)$.
\end{notation}

\subsection{The local \texorpdfstring{$\A^1$}{A1}-degree}\label{subsec:local-a1-deg}
Given an endomorphism of affine space $f: \A^n_k \to \A^n_k$ with an isolated zero at a point $p$, we describe how to obtain an endomorphism of the sphere spectrum in the stable motivic homotopy category $\SH(k)$, following the exposition of \cite[pp. 438--439]{KW-EKL}. We remind the reader of Conventions~\ref{conventions}, which we use in what follows.

As $p$ is an isolated zero of $f$, we may find an open neighborhood $U \subseteq \A^n_k$ for which $f^{-1}(0) \cap U = \{p\}$, that is, an open neighborhood containing no other zeros of $f$. Viewing $U \subseteq \A^n_k \subseteq \P^n_k$ as an open subset of projective space via any choice of affine chart on projective space, we may take a Nisnevich-local pushout diagram in $\Spc_k^{\A^1}$:
\[ \begin{tikzcd}
    U \minus \{p\} \rar\dar & \P^n_k \minus \{p\}\dar \\
    U\rar & \P^n_k.\po
\end{tikzcd} \]
This induces an $\A^1$-weak equivalence on cofibers
\begin{align*}
    \frac{U}{U \minus \{p\}} \xto{\sim} \frac{\P^n_k}{\P^n_k \minus \{p\}}.
\end{align*}
Our projective space is endowed with a canonical trivialization arising from the trivialization of affine space. Thus by the purity theorem \cite[Theorem~2.23]{MV}, we obtain canonical $\A^1$-weak equivalences with the Thom space of the trivial rank $n$ bundle over $\Spec k(p)$, and with the following cofiber of projective space:
\begin{align*}
   \frac{U}{U\minus\{p\}}\xto{\sim} \frac{\P^n_k}{\P^n_k \minus \{p\}} \simeq \Th(\O_{k(p)}^n) \simeq \left( \P^n_k / \P^{n-1}_k \right) \smashprod \Spec k(p)_+.
\end{align*}
We remark that $U$ was chosen to satisfy $f(U \minus \{p\}) \subseteq \A^n_k \minus \{0\}$, and we can perform a completely analogous procedure to obtain $\A^1$-weak equivalences
\begin{align*}
    \frac{\A^n_k}{\A^n_k \minus \{0\}} \xto{\sim} \frac{\P^n_k}{\P^n_k \minus \{0\}} \simeq \Th(\O_{k}^n)\simeq \P^n_k / \P^{n-1}_k.
\end{align*}
Recall that in differential topology, the local degree is defined as the homotopy class of an induced map of spheres about a point. The space $\P^n_k/\P^{n-1}_k$ analogously plays the role of a sphere in $\Spc^{\A^1}_k$ when constructing the local $\A^1$-degree. 

\begin{definition}\label{def:collapse}The \textit{collapse map} is the map  $c_p : \P^n_k/\P^{n-1}_k \to \P^n_k/(\P^n_k \minus \{p\})$ induced by the inclusion $\P^{n-1}_k \subseteq \P^n_k \minus \{p\}$. \end{definition} 

\begin{definition}\label{defn:htpy-class-of-maps-for-isolated-zero} For any $f: \A^n_k \to \A^n_k$ with an isolated zero at $p$, we denote by $f_p \in \left[\P^n_k/\P^{n-1}_k, \P^n_k/\P^{n-1}_k\right]_{\A^1}$ the $\A^1$-homotopy class of maps in the unstable motivic homotopy category assigned to the composite
\[
    \P^n_k/\P^{n-1}_k \xto{c_p} \frac{\P^n_k}{\P^n_k \minus \{p\}} \xfrom{\sim} \frac{U}{U\minus\{p\}} \xto{\bar{f}} \frac{\A^n_k}{\A^n_k \minus \{0\}} \xto{\sim} \P^n_k/\P^{n-1}_k.
\]
\end{definition}

\begin{remark} 
When $p$ is $k$-rational, one can avoid the collapse map by applying the purity theorem to obtain the composite
\[
    \P^n_k/\P^{n-1}_k \simeq \frac{U}{U\minus\{p\}}\xto{\bar{f}} \frac{\A^n_k}{\A^n_k \minus \{0\}} \xto{\sim} \P^n_k/\P^{n-1}_k.
\]
This composite yields the same element of $\left[\P^n_k/\P^{n-1}_k, \P^n_k/\P^{n-1}_k\right]_{\A^1}$ as in Definition~\ref{defn:htpy-class-of-maps-for-isolated-zero}. Indeed, by \cite[Lemma~10]{KW-EKL}, the composite of the collapse map with the canonical $\A^1$-weak equivalence $\P^n_k/(\P^n_k\minus\{0\}) \xto{\sim} \P^n_k/\P^{n-1}_k$ is the class of the identity map in $\left[\P^n_k/\P^{n-1}_k, \P^n_k/\P^{n-1}_k\right]_{\A^1}$.
\end{remark}

\begin{remark}\label{rmk:deg-independent-of-U} The $\A^1$-homotopy class of $f_p$ does not depend upon the original choice of open neighborhood $U \ni p$, provided that $U$ contains no other zeros of $f$ besides $p$. This follows immediately from our ability to provide an $\A^1$-weak equivalence between the cofiber $U/\left(U\minus \{p\}\right)$ and the Thom space $\Th(\O_k^n)$.
\end{remark}

We now describe how to obtain an endomorphism of the sphere spectrum in $\SH(k)$ from the class $f_p$ defined above. By \cite[Proposition~2.17]{MV}, we have a canonical $\A^1$-weak equivalence
\begin{equation}\label{eqn:cofiber-projective-space-is-sphere}
\begin{aligned}
    \P^n_k / \P^{n-1}_k \simeq \left(\P^1_k\right)^{\smashprod n}.
\end{aligned}
\end{equation}
We recall also that $\P^1 \simeq S^1 \smashprod \G_m$ as elements of the stable homotopy category. In particular by following the indexing convention of \cite{Morel-Trieste} for motivic spheres, we see that $\Sigma^\infty \P^1_k = \Sigma^{2,1} \mathbf{1}_k$ in $\SH(k)$, where $\mathbf{1}_k$ denotes the sphere spectrum. We therefore have that 
\[\Sigma^\infty  \P^n_k / \P^{n-1}_k \simeq \Sigma^{2n,n} \mathbf{1}_k\]
in $\SH(k)$. It is immediate that, by desuspending, we obtain a canonical isomorphism
\[\End_{\SH(k)} \left( \Sigma^{2n,n} \mathbf{1}_k \right) \cong \End_{\SH(k)}(\mathbf{1}_k)\]
in the stable homotopy category. Collecting these facts together, we see that $f_p$ determines an element in $\End_{\SH(k)}(\mathbf{1}_k)$. Abusing notation, we will refer to this endomorphism of the sphere spectrum as $f_p$.

\begin{theorem}[Morel]\label{thm:Morel-deg}
 For any field $k$, there is an isomorphism
\begin{equation}\label{eqn:morel-deg-iso}
\begin{aligned}
    \deg^{\A^1}: \End_{\SH(k)}(\mathbf{1}_k) \cong \GW(k).
\end{aligned}
\end{equation}
Morel initially required the assumption that $k$ be perfect \cite{Morel-book}, however this can be removed via work of Hoyois \cite[Appendix~A]{Hoy13}.
\end{theorem}

\begin{definition} 
With notation as above, the image of $f_p$ in $\GW(k)$ under $\deg^{\A^1}$ is the \textit{local $\A^1$-degree} of $f$ at $p$, denoted $\deg_p^{\A^1} (f)$.
\end{definition}

\subsection{Stable motivic homotopy theory}\label{subsec:SHK}
We begin by recalling a few concepts and results from stable motivic homotopy theory that will play a role in the proof of Theorem~\ref{thm:EKL-is-deg}. 

The category theory of $\SH(k)$ supports a six functor formalism, the general exposition of which we defer to \cite[\S2]{hoyois}. Indeed, for the purposes of this paper, we need only consider this formalism in the case of functors induced by maps $\rho: \Spec k(p) \to \Spec k$, where $k(p)/k$ is a finite separable field extension. We recall that we have an adjunction
\begin{align*}
    \rho^\ast : \SH(k) \rightleftarrows \SH(k(p)) : \rho_\ast.
\end{align*}
Since $\rho$ is separated and finite type, we also have an exceptional adjunction
\begin{align*}
    \rho_! : \SH(k(p)) \rightleftarrows \SH(k) : \rho^!.
\end{align*}
We denote by $\eta$ the unit of the adjunction between the direct and inverse image functors, and by $\epsilon$ the counit of the exceptional adjunction. That is, we have natural transformations: 
\begin{align*}
    \eta: \id_{\SH(k)} \Rightarrow \rho_\ast \rho^\ast \\
    \epsilon: \rho_! \rho^! \Rightarrow \id_{\SH(k)}.
\end{align*}
\begin{remark}\label{rmk:six-functors} 
To facilitate exposition, we pause here to provide references for a few basic facts about six functors which we will make use of in this paper. Let $\rho$ be as in Conventions~\ref{conventions}.
\begin{enumerate}
    \item Since $\rho$ is smooth, $\rho^\ast$ admits a left adjoint, denoted $\rho_\sharp$. As $\rho$ is furthermore finite and \'etale, we have a canonical equivalence $\rho_\ast \simeq \rho_\sharp$ \cite[p.21]{hoyois}.
    \item We have a canonical isomorphism $\rho_\ast\rho^\ast \mathbf{1}_{k} \simeq \rho_\ast \mathbf{1}_{k(p)}$. This is due to \cite[p.112,~Proposition~2.17(3)]{MV}. See also \cite[Equation~11]{KW-EKL}.
    \item Under our assumptions on $\rho$, we have canonical natural isomorphisms $\rho^! \simeq \rho^\ast$ and $\rho_! \simeq \rho_\sharp \simeq \rho_\ast$. This may be found in \cite[p.21]{hoyois}. In particular, we remark that $\rho_\sharp$ can be interpreted as a forgetful functor under the structure map $\rho$.
    \item There is a canonical equivalence $\rho_! \mathbf{1}_{k(x)} \cong \Spec k(x)_+$ in $\SH(k)$. See \cite[p.441]{KW-EKL}.
    \end{enumerate}
\end{remark}

We are now in a position to recall a description of the collapse map at the level of the stable motivic homotopy category.

\begin{lemma}\label{lem:collapse-map} \cite{hoyois,KW-EKL} In the stable homotopy category $\SH(k)$, the collapse map of Definition~\ref{def:collapse}
\begin{align*}
    c_p: \P_k^{n}/\P_k^{n-1} \to \P_k^{n}/(\P_k^{n}\minus\{p\}) \cong \left(\P^n_{k} / \P^{n-1}_{k}\right) \smashprod \Spec(k(p))_+
\end{align*}
is computed by applying $\P_k^{n}/\P_k^{n-1} \wedge(-)$ to the component of the unit $\eta$ at the sphere spectrum:
\begin{align*}
    \eta_{\mathbf{1}_k}: \textbf{1}_k \to \rho_\ast \rho^\ast \textbf{1}_k \cong \rho_\ast \textbf{1}_{k(p)}.
\end{align*}
\end{lemma}
\begin{proof}
The case $n=1$ may be found in \cite[Lemma~5.5]{hoyois}, and the proof generalizes to higher $n$ as in \cite[Lemma~13]{KW-EKL}.
\end{proof}

\begin{remark} 
We can furthermore describe $f_p \in \End_{\SH(k)}(\mathbf{1}_k)$ in the following way. Recall that $f$ induces a map
\[\bar{f}: U/(U\minus\{p\}) \to \A^n_k/(\A^n_k\minus\{0\}).\]
As above, we have $\A^1$-weak equivalences
\[U/(U\minus\{p\}) \simeq \left( \P^n_k/\P^{n-1}_k \right) \smashprod \Spec(k(p))_+\]
and 
\[ \A^n_k/(\A^n_k\minus\{0\}) \simeq \P^n_k/\P^{n-1}_k.\]
We may thus identify $\bar{f}$ with the composite $\left( \P^n_k/\P^{n-1}_k \right) \smashprod \Spec(k(p))_+ \to \P^n_k/\P^{n-1}_k$ in $\SH(k)$. By Definition~\ref{defn:htpy-class-of-maps-for-isolated-zero}, we have that $f_p$ is the composite of $\bar{f}$ and the collapse map. In $\SH(k)$, we can record this via the following commutative diagram:
\[ \begin{tikzcd}
    {\left( \P^n_k/\P^{n-1}_k \right) \smashprod \Spec(k(p))_+}\rar["\bar{f}" above] &\P^n_k/\P^{n-1}_k \\
    \P^n_k/\P^{n-1}_k\uar["{\P_k^{n}/\P_k^{n-1} \wedge(\eta_{\mathbf{1}_k})}" left]\ar[ur,bend right=20, "f_p" below right] & \\
\end{tikzcd} \]
\end{remark}
We now recall that the trace $\Tr_{k(p)/k}: \GW(k(p)) \to \GW(k)$ can be described purely in terms of maps in the motivic homotopy category, under the isomorphism of Theorem~\ref{thm:Morel-deg}. 

\begin{definition}\label{def:trace-stable-htpy} 
The \textit{transfer} 
\[
    \Tr_{k(p)/k} : \End_{\SH(k(p))}(\mathbf{1}_{k(p)}) \to \End_{\SH(k)}(\mathbf{1}_k)
\]
is defined by sending $\omega \in \End_{\SH(k(p))}(\mathbf{1}_{k(p)})$ to the composite
\begin{align*}
    \mathbf{1}_k \xto{\eta_{\mathbf{1}_k}} \rho_\ast \mathbf{1}_{k(p)} \simeq \rho_\sharp \mathbf{1}_{k(p)}\xto{\rho_\sharp \omega} \rho_\sharp \mathbf{1}_{k(p)} \simeq \rho_! \rho^! \mathbf{1}_k \xto{\epsilon_{\mathbf{1}_k}} \mathbf{1}_k.
\end{align*}
\end{definition}

\begin{lemma}\label{lem:hoyois-trace}\cite[Proposition~5.2,~Lemma~5.3]{hoyois} 
The transfer agrees with the field trace. That is, the diagram 
\[ \begin{tikzcd}
    \End_{\SH(k(p))} \left( \mathbf{1}_{k(p)} \right)\rar["\Tr_{k(p)/k}" above]\dar["\cong" left] & \End_{\SH(k)} \left( \mathbf{1}_{k} \right)\dar["\cong" right]  \\
    \GW(k(p))\rar["\Tr_{k(p)/k}" below] & \GW(k)\\
\end{tikzcd} \]
commutes, where the vertical maps are given by Morel's degree isomorphism (Equation~\ref{eqn:morel-deg-iso}), the top map is the transfer (Definition~\ref{def:trace-stable-htpy}), and the bottom map is the trace map on the Grothendieck--Witt group of $k$ (Definition~\ref{def:trace-GW}).
\end{lemma}

\subsection{The Scheja--Storch bilinear form}\label{subsec:SS}
We give a brief description of the Scheja--Storch bilinear form (see also \cite{SchejaStorch}, \cite{KW-EKL}, and \cite[Section 4]{KW17}). Given a polynomial map $$f=(f_1,\dots,f_n):\A^n_k\to\A^n_k$$ with isolated zero $p$, let $\mf{m}$ be the maximal ideal of $k[x_1,\ldots,x_n]$ corresponding to the point $p$. Consider the local algebra $Q_p=\frac{k[x_1,\ldots,x_n]_{\mf{m}}}{(f_1,\ldots,f_n)}$. As a local complete intersection, $Q_p$ is isomorphic to its dual $\Hom_k(Q_p,k)$. Scheja and Storch construct an explicit $Q_p$-linear isomorphism $\Theta:\Hom_k(Q_p,k)\to Q_p$ realizing this self-duality, which gives us a distinguished homomorphism $\eta:=\Theta^{-1}(1):Q_p\to k$.

\begin{definition}
Given a polynomical function $f: \A^n_k \to \A^n_k$, the {\it Scheja--Storch bilinear form} $\beta_p(f):Q_p\times Q_p\to k$ is given by $\beta_p(f)(x,y)=\eta(xy)$, where $\eta$ is defined in the preceeding paragraph.
\end{definition}
Since $Q_p$ is commutative, the Scheja--Storch bilinear form is symmetric. By \cite[Lemma 28]{KW-EKL} and \cite[Proposition 3.4]{EisenbudLevine}, the Scheja--Storch form is non-degenerate. Thus the Scheja--Storch form gives a class in $\GW(k)$, which we denote by $\ind_p(f)$ .

Kass--Wickelgren show that if $p$ is $k$-rational or if $f$ is \'etale at $p$, then $\deg^{\A^1}_p(f)=\ind_p(f)$ \cite{KW-EKL}. They  also show that if $p$ is a finite separable point, then $$\ind_p(f)=\Tr_{k(p)/k}\ind_{\til{p}}(f_{k(p)}),$$ where $\til{p}$ is the canonical $k(p)$-point above $p$ and $f_{k(p)}:\A^n_{k(p)}\to\A^n_{k(p)}$ is the lift of $f$ \cite[Proposition 32]{KW17}. Given these two results, one would expect Theorem~\ref{thm:EKL-is-deg} to be true.

\section{Main Results}\label{sec:main-proof}
We now proceed to the proof of the main theorem, as stated in Theorem~\ref{thm:EKL-is-deg}. Our first step is to apply the machinery of Subsection~\ref{subsec:SHK} to frame our problem in terms of motivic homotopy theory. 
Recall from Conventions~\ref{conventions} that we have already fixed a choice of field embedding $k \hookto k(p)$. Thus, for any $k$-scheme $X$ and any point $p \in X$, we have the canonical $k(p)$-rational point $\til{p} \in X_{k(p)}$ sitting over $p$, defined via the following pullback diagram:
\begin{align*}
    \begin{tikzcd}[ampersand replacement =\&]
    \Spec k(p)\ar[dr,dashed,"\til{p}" above right]\ar[drr,bend left=20,"\id" above right]\ar[ddr,bend right=20, "p" below left] \&  \&  \\
     \& X_{k(p)}\dar\rar\pb \& \Spec k(p) \dar["\rho" right] \\
     \& X\rar \& \Spec k.
    \end{tikzcd}
\end{align*}
We write $f_{k(p)} = f\otimes_k k(p)$ and $\pi: \A^n_{k(p)} \to \A^n_k$ for the morphisms induced by base change.
We then may consider the following diagram of $k$-schemes:
\begin{align*}
	\begin{tikzcd}[ampersand replacement =\&]
	\A^n_{k(p)}\rar["f_{k(p)}" above]\dar["\pi" left] \& \A^n_{k(p)}\dar["\pi" right] \\
	\A^n_k \rar["f" below] \& \A^n_k 
	\end{tikzcd} \quad\text{which maps}\quad
	\begin{tikzcd}[ampersand replacement =\&]
	\til{p} \rar[maps to] \dar[maps to] \& 0\dar[maps to] \\
	p\rar[maps to] \& 0.
	\end{tikzcd}
\end{align*}
Note that the point $\til{p}$ is a root of $f_{k(p)}$, so $f_{k(p)}$ has an isolated rational zero at $\til{p}$. Let $U \subseteq \A^n_{k(p)}$ be an open neighborhood containing $\til{p}$ and no other zeros of $f_{k(p)}$. As the structure map $\A^n_k \to \Spec k$ is universally open, $\pi$ is an open morphism of schemes. Thus, $\pi(U)$ is an open neighborhood of $p$ and contains no other zeros of $f$ by construction. Taking cofibers, we obtain an induced diagram of motivic spaces
\begin{equation}\label{eqn:local-diagram-1}
    \begin{tikzcd}
    \frac{U}{U \minus \left\{ \til{p} \right\}} \rar["\bar{f}_{k(p)}" above]\dar["\bar{\pi}_p" left]  & \frac{\A^n_{k(p)}}{\A^n_{k(p)} \minus \left\{ 0 \right\}} \dar["\bar{\pi}_0" right] \\
    \frac{\pi(U)}{\pi(U) \minus \left\{ p \right\}} \rar["\bar{f}" below]  & \frac{\A^n_k}{\A^n_k\minus \left\{ 0 \right\}}.
    \end{tikzcd}
\end{equation}
As discussed in Section~\ref{subsec:local-a1-deg}, we have the following $\A^1$-weak equivalences:
\begin{align*}
    \frac{U}{U \minus \{\til{p}\}} \xto{\sim} \frac{\P^n_{k(p)}}{\P^n_{k(p)} \minus \left\{ \til{p} \right\}} &\simeq \left(\P^n_k / \P^{n-1}_k\right)\smashprod \Spec(k(p))_+, \\
    \frac{\A^n_{k(p)} }{\A^n_{k(p)} \minus \left\{ 0 \right\}} \xto{\sim} \frac{\P^n_{k(p)}}{\P^n_{k(p)} \minus \left\{ 0 \right\}} &\simeq \left(\P^n_k / \P^{n-1}_k\right)\smashprod \Spec(k(p))_+, \\
    \frac{\pi(U)}{\pi(U) \minus \left\{ p \right\}} \xto{\sim} \frac{\P^n_k}{\P^n_k \minus \left\{ p \right\}} &\simeq \left(\P^n_k / \P^{n-1}_k\right)\smashprod \Spec(k(p))_+, \\
    \frac{\A^n_k}{\A^n_k \minus \left\{ 0 \right\}} \xto{\sim} \frac{\P^n_k}{\P^n_k \minus \left\{ 0 \right\}}  &\simeq \P^n_k/\P^{n-1}_k. 
\end{align*}
Appending these $\A^1$-weak equivalences to Diagram~\ref{eqn:local-diagram-1}, we obtain the following diagram in the unstable homotopy category $\mathcal{H}(k)$.
\begin{equation}\label{eqn:local-diagram-2}
    \begin{tikzcd}[column sep=tiny]
    & \left(\P^n_k / \P^{n-1}_k\right)\smashprod \Spec(k(p))_+ \ar[rr,dashed]\ar[dd,dashed]  & & \left(\P^n_k / \P^{n-1}_k\right)\smashprod \Spec(k(p))_+ \ar[dd,dashed] \\
   \frac{U}{U \minus \{\til{p}\}}\ar[ur,"\sim" above left]\ar[dd,"\bar{\pi}_p" left]\ar[rr,crossing over,"\bar{f}_{k(p)}" near start] & & \frac{\A^n_{k(p)}}{\A^n_{k(p)} \minus \{0\}}\ar[ur,"\sim" above left] \\
    & \left(\P^n_k / \P^{n-1}_k\right)\smashprod \Spec(k(p))_+ \ar[rr,dashed]  & &\P^n_k/\P^{n-1}_k \\
    \frac{\pi(U)}{\pi(U) \minus \{p\}}\ar[ur,"\sim" above left]\ar[rr,"\bar{f}" below]  & & \frac{\A^n_k}{\A^n_k \minus \{0\}}\ar[ur,"\sim" above left]\ar[from=uu,crossing over,"\bar{\pi}_0" near start]  \\
    \end{tikzcd}
\end{equation}
The dashed arrows above are obtained by inverting $\A^1$-weak equivalences.

We now turn our attention to the dashed face of the cube \eqref{eqn:local-diagram-2}. We obtain the class $f_p$ from Definition~\ref{defn:htpy-class-of-maps-for-isolated-zero} by pre-composing with the collapse map (see Diagram~\ref{eqn:local-diagram-3}). Working in the stable homotopy category, the top edge of the dashed face is exactly the image of $\left( f_{k(p)} \right)_{\til{p}}$ under $\SH(k(p))\to\SH(k)$. Taking suspension spectra, we get the following diagram in $\SH(k)$.

\begin{equation}\label{eqn:local-diagram-3}
    \begin{tikzcd}[column sep=large,row sep=large]
    \left(\P^n_k / \P^{n-1}_k\right)\smashprod \Spec(k(p))_+ \ar[r,"\rho_\ast (f_{k(p)})_{\til{p}}" above]\ar[d,swap,"r"]  & \left(\P^n_k / \P^{n-1}_k\right)\smashprod \Spec(k(p))_+ \ar[d,"g" right] \\
    \left(\P^n_k / \P^{n-1}_k\right)\smashprod \Spec(k(p))_+ \ar[r] &\P^n_k/\P^{n-1}_k \\
    \P^n_k/\P^{n-1}_k\uar["\left( \P^n_k/\P^{n-1}_k \right)\smashprod \eta_{\mathbf{1}_k}" left]\ar[ur, bend right=20,"f_p" below right]
    \end{tikzcd}
\end{equation}
The rest of our paper will center around the diagram above. In proving Theorem~\ref{thm:EKL-is-deg}, we will show that $r$ in Diagram~\eqref{eqn:local-diagram-3} is invertible in $\SH(k)$, which allows us to rewrite $f_p$ by exploiting the commutativity of this diagram.
 
\subsection{The stable classes of $r$ and $g$}
In order to analyze Diagram~\eqref{eqn:local-diagram-3} we state the following two lemmas, which allow us to characterize the $\SH(k)$-classes of $r$ and $g$, respectively.

\begin{lemma}\label{lem:stable-class-of-r}
Let $p\in \A^n_k$ be a closed point with finite separable residue field $k(p)/k$. Let $\pi:\A^n_{k(p)}\to \A^n_k$ be the projection map induced by the structure map $\rho: \Spec k(p) \to \Spec k$, and let $\til{p}\in \A^n_{k(p)}$ be the canonical $k(p)$-rational point above $p$. Then for any open neighborhood $U$ about $\til{p}$ such that $U\cap \pi^{-1}(p) = \{\til{p}\}$, the stable class in $\SH(k)$ of the map
\begin{equation}\label{eqn:thom-map-about-p}
\frac{U}{U\minus\{\til{p}\}}\to\frac{\pi(U)}{\pi(U)\minus\{p\}}
\end{equation}
is given by $(\P^n_k/\P^{n-1}_k)\wedge\rho_*\id_{\mathbf{1}_{k(p)}}$.
\end{lemma}

\begin{proof}
The base change $\pi:\A^n_{k(p)}\to\A^n_k$ is simply $\Spec$ applied to the $k$-algebra homomorphism $\iota:k[x_1,...,x_n]\hookto k(p)[x_1,...,x_n]$. As $\iota(x_i)=x_i$, we get an induced map $T\A^n_{k(p)}\to\pi^*T\A^n_k$ which in turn induces an isomorphism $(T\A^n_{k(p)})_{\til{p}}\xto{\sim}(\pi^*T\A^n_k)_{\til{p}}$. 
The right hand side is easily seen to be $T_p\A^n_k\otimes_k k(p)$. As $k(p)$-vector spaces, we have isomorphisms $T_{\til{p}}\A^n_{k(p)}\cong T_p\A^n_k\otimes_k k(p)\cong\A^n_{k(p)}$. 

Next, we consider the Thom spaces $\Th(T_{\til{p}}\A^n_{k(p)})$ and $\Th(T_p\A^n_k\otimes_k k(p))$. Via the purity isomorphism \cite[Theorem~2.23]{MV}, we have weak equivalences of $k$-motivic spaces
\begin{align*}
    \frac{U}{U \minus \{\til{p}\}} &\simeq \Th \left( T_{\til{p}} \A^n_{k(p)} \right)\xto{\sim} \left(\P^n_k / \P^{n-1}_k\right)\smashprod \Spec(k(p))_+,  \\
   \frac{\pi(U)}{\pi(U) \minus \{p\}} &\simeq \Th \left( T_p \A^n_k \right)\xto{\sim} \left(\P^n_k / \P^{n-1}_k\right)\smashprod \Spec(k(p))_+.
\end{align*}
Since base change is a left adjoint and Thom spaces are obtained by taking colimits, we deduce 
\[ \Th(T_p\A^n_k\otimes_k k(p))\simeq\rho^\ast\Th(T_p\A^n_k)\]
as $k(p)$-motivic spaces. Moreover, purity implies that $\Th(T_p \A^n_k)$ is a $k(p)$-motivic space, so its base change to $k(p)$ is canonically identified with itself in the homotopy category of $k(p)$-motivic spaces. That is, the class in $\SH(k(p))$ of Equation~\ref{eqn:thom-map-about-p} is given by the class in $\SH(k(p))$ of the canonical $\A^1$-weak equivalence $\Th(T_{\til{p}}\A^n_{k(p)})\xto{\sim} \Th(T_p\A^n_k).$ This has the class of $$(\P^n_{k(p)}/\P^{n-1}_{k(p)})\wedge\id_{\mathbf{1}_{k(p)}}$$ in $\SH(k(p))$ under the identification given by purity. Finally, we push forward via the functor $\rho_\ast : \SH(k(p)) \to \SH(k)$ to get that the class of Equation~\ref{eqn:thom-map-about-p} is $(\P^n_{k}/\P^{n-1}_{k})\wedge\rho_\ast \id_{\mathbf{1}_{k(p)}}$. 
\end{proof}

\begin{lemma}\label{lem:stable-class-of-g}
Let $k(p)/k$ be a finite separable field extension, let $q\in \A^n_k$ be any $k$-rational point, and let $\pi:\A^n_{k(p)}\to \A^n_k$ be the projection map induced by the structure map $\rho: \Spec k(p) \to \Spec k$. Denote by $\til{q}\in \A^n_{k(p)}$ the canonical $k(p)$-rational point above $q$. Then for any open neighborhood $U$ containing $\til{q}$, the stable class in $\SH(k)$ of the map
\begin{equation}\label{eqn:thom-map-about-q}
\frac{U}{U\minus\{\til{q}\}}\to\frac{\pi(U)}{\pi(U)\minus\{q\}}
\end{equation}
is given by $(\P^n_k/\P^{n-1}_k)\wedge\epsilon_{\mathbf{1}_k}$.
\end{lemma}
\begin{proof}
Since $q$ is $k$-rational, $\til{q}$ is the unique canonical $k(p)$-rational point above $q$, so we may assume that $\til{q}$ and $q$ are the origins of $\A^n_{k(p)}$ and $\A^n_k$, respectively, in which case we may take $U = \A^n_{k(p)}$. It thus suffices to consider the class in $\SH(k)$ of the map of Thom spaces $\Th_0 \A^n_{k(p)} \to \Th_0 \A^n_k$ induced by the canonical map $\pi: \A^n_{k(p)} \to \A^n_k$. By viewing affine space as a trivial vector bundle over the origin, \cite[p.9]{hoyois} implies that this is the desired component of the counit of the exceptional adjunction.
\end{proof}

\begin{corollary}\label{cor:g and r}
The map $g$ in Diagram~\ref{eqn:local-diagram-3} is $\left(\P^n_k /\P^{n-1}_k\right) \smashprod \epsilon_{\mathbf{1}_k}$, and the map $r$ is $\left(\P^n_k /\P^{n-1}_k\right) \smashprod \rho_\ast\id_{\mathbf{1}_{k(p)}}$.
\end{corollary}

\begin{remark} 
Lemmas~\ref{lem:stable-class-of-r} and~\ref{lem:stable-class-of-g} hold more generally for schemes that are locally isomorphic to affine space, as we rely only on local computations in their proofs.
\end{remark}

\subsection{Proof of Theorem~\ref{thm:EKL-is-deg}}
By Corollary~\ref{cor:g and r}, we may rewrite Diagram~\eqref{eqn:local-diagram-3} as
\begin{equation}\label{eqn:local-diagram-with-transfers}
    \begin{tikzcd}[column sep=large,row sep=large]
    \left(\P^n_k / \P^{n-1}_k\right)\smashprod \Spec(k(p))_+ \ar[r,"\rho_\ast (f_{k(p)})_{\til{p}}" above]\ar[d,"\left( \P^n_k/\P^{n-1}_k\right) \smashprod \rho_\ast \id_{\mathbf{1}_{k(p)}}" left, "\sim" right]  & \left(\P^n_k / \P^{n-1}_k\right)\smashprod \Spec(k(p))_+ \ar[d,"\left( \P^n_k/\P^{n-1}_k \right)\smashprod \epsilon_{\mathbf{1}_k}" right] \\
    \left(\P^n_k / \P^{n-1}_k\right)\smashprod \Spec(k(p))_+ \ar[r] &\P^n_k/\P^{n-1}_k \\
    \P^n_k/\P^{n-1}_k\uar["\left( \P^n_k/\P^{n-1}_k \right)\smashprod \eta_{\mathbf{1}_k}" left]\ar[ur, bend right=20,"f_p" below right]
    \end{tikzcd}
\end{equation}

In the stable homotopy category $\SH(k)$, Diagram~\ref{eqn:local-diagram-with-transfers} is $\P^n_k /\P^{n-1}_k \smashprod (-)$ applied to the diagram
\begin{equation}\label{eqn:diagram-of-sphere-spectra}
    \begin{tikzcd}[column sep=large,row sep=large]
    \rho_\ast \mathbf{1}_{k(p)} \ar[r,"\rho_\ast (f_{k(p)})_{\til{p}}" above]\ar[d," \rho_\ast \id_{\mathbf{1}_{k(p)}}" left, "\sim" right]  &  \rho_\ast \mathbf{1}_{k(p)} \ar[d,"\epsilon_{\mathbf{1}_k}" right] \\
    \rho_\ast \mathbf{1}_{k(p)} \ar[r] & \mathbf{1}_k \\
    \mathbf{1}_k\uar["\eta_{\mathbf{1}_k}" left]\ar[ur, bend right=20,"f_p" below right]
    \end{tikzcd}
\end{equation}

We remark that we are able to invert the weak equivalence $\rho_*\id_{\mathbf{1}_{k(p)}}$ of Diagram~\ref{eqn:diagram-of-sphere-spectra}. Since Diagram~\ref{eqn:diagram-of-sphere-spectra} is commutative, we may express $f_p$ as the composite
\begin{align*}
    \mathbf{1}_k \xto{\eta_{\mathbf{1}_k}} \rho_\ast \rho^\ast \mathbf{1}_{k} \simeq \rho_\ast \mathbf{1}_{k(p)} \xto{\rho_\ast \left( f_{k(p)} \right)_{\til{p}}} \rho_\ast \mathbf{1}_{k(p)} \simeq \rho_! \rho^! \mathbf{1}_k \xto{\epsilon_{\mathbf{1}_k}} \mathbf{1}_k.
\end{align*}
Recall that in the setting of Theorem~\ref{thm:EKL-is-deg}, the morphism $\rho$ is finite and \'etale, which gives a canonical isomorphism $\rho_\ast\simeq\rho_\sharp$ (Remark~\ref{rmk:six-functors}). Thus by Definition~\ref{def:trace-stable-htpy}, we have $f_p = \Tr_{k(p)/k} \left( f_{k(p)} \right)_{\til{p}}$. Applying Lemma~\ref{lem:hoyois-trace}, we conclude that $\deg_{p}^{\A^1} f = \Tr_{k(p)/k} \deg_{\til{p}}^{\A^1} f_{k(p)}$,
as desired.

\subsection{A brief proof of Corollary~\ref{cor:SS-is-local-deg}}

In \cite[Proposition~32]{KW17}, the authors prove that the Scheja--Storch bilinear form, denoted $\ind_p f$, is computed by the trace $\ind_p f = \Tr_{k(p)/k} \ind_{\til{p}} f_{k(p)}$. Moreover in \cite{KW-EKL}, the authors prove that at any rational point, the Scheja--Storch form agrees with the local $\A^1$-degree. Combining these two results with Theorem~\ref{thm:EKL-is-deg}, for any isolated zero $p$ with finite separable residue field we have that
\begin{align*}
    \ind_p f &= \Tr_{k(p)/k} \ind_{\til{p}} f_{k(p)} = \Tr_{k(p)/k} \deg_{\til{p}}^{\A^1} f_{k(p)} = \deg_p^{\A^1} f.
\end{align*}

\bibliographystyle{amsalpha}
\bibliography{citations.bib}{}

\providecommand{\bysame}{\leavevmode\hbox to3em{\hrulefill}\thinspace}
\providecommand{\MR}{\relax\ifhmode\unskip\space\fi MR }
\providecommand{\MRhref}[2]{%
  \href{http://www.ams.org/mathscinet-getitem?mr=#1}{#2}
}
\providecommand{\href}[2]{#2}
\begin{thebibliography}{{Hoy}15b}

\bibitem[AE16]{antieau2016primer}
Benjamin Antieau and Elden Elmanto, \emph{A primer for unstable motivic
  homotopy theory}, 2016.

\bibitem[Ayo07]{Ayoub}
Joseph Ayoub, \emph{Les six op\'{e}rations de {G}rothendieck et le formalisme
  des cycles \'{e}vanescents dans le monde motivique. {I}}, Ast\'{e}risque
  (2007), no.~314, x+466 pp. (2008). \MR{2423375}

\bibitem[BKW18]{Bethea}
Candace Bethea, Jesse~Leo Kass, and Kirsten Wickelgren, \emph{An example of
  wild ramification in an enriched {R}iemann-{H}urwitz formula}, arXiv preprint
  arXiv:1812.03386 (2018).

\bibitem[CD19]{Cisinski}
Denis-Charles Cisinski and Fr\'{e}d\'{e}ric D\'{e}glise, \emph{Triangulated
  categories of mixed motives}, Springer Monographs in Mathematics, Springer,
  Cham, [2019] \copyright 2019. \MR{3971240}

\bibitem[EL77]{EisenbudLevine}
David Eisenbud and Harold~I. Levine, \emph{An algebraic formula for the degree
  of a {$C^{\infty }$} map germ}, Ann. of Math. (2) \textbf{106} (1977), no.~1,
  19--44, With an appendix by Bernard Teissier, ``Sur une in\'{e}galit\'{e} \`a
  la Minkowski pour les multiplicit\'{e}s''. \MR{0467800}

\bibitem[Hoy15a]{Hoy13}
Marc Hoyois, \emph{From algebraic cobordism to motivic cohomology}, J. Reine
  Angew. Math. \textbf{702} (2015), 173--226. \MR{3341470}

\bibitem[{Hoy}15b]{hoyois}
Marc {Hoyois}, \emph{A quadratic refinement of the
  {G}rothendieck--{L}efschetz--{V}erdier trace formula}, Algebraic \& Geometric
  Topology \textbf{14} (2015), no.~6, 3603--3658.

\bibitem[KW17]{KW17}
Jesse~Leo {Kass} and Kirsten {Wickelgren}, \emph{{An arithmetic count of the
  lines on a smooth cubic surface}}, ArXiv e-prints (2017).

\bibitem[KW19]{KW-EKL}
Jesse~Leo Kass and Kirsten Wickelgren, \emph{The class of
  {E}isenbud-{K}himshiashvili-{L}evine is the local {$\mathbf{A}^1$}-{B}rouwer
  degree}, Duke Math. J. \textbf{168} (2019), no.~3, 429--469. \MR{3909901}

\bibitem[Lev17]{Levine}
Marc Levine, \emph{Toward an enumerative geometry with quadratic forms}, arXiv
  preprint arXiv:1703.03049 (2017).

\bibitem[LV19]{VogtLarson}
Hannah Larson and Isabel Vogt, \emph{An enriched count of the bitangents to a
  smooth plane quartic curve}, arXiv preprint arXiv:1909.05945 (2019).

\bibitem[Mor04]{Morel-Trieste}
Fabien Morel, \emph{An introduction to {$\mathbb{A}^1$}-homotopy theory},
  Contemporary developments in algebraic {$K$}-theory, ICTP Lect. Notes, XV,
  Abdus Salam Int. Cent. Theoret. Phys., Trieste, 2004, pp.~357--441.
  \MR{2175638}

\bibitem[Mor12]{Morel-book}
\bysame, \emph{{$\mathbb{A}^1$}-algebraic topology over a field}, Lecture Notes
  in Mathematics, vol. 2052, Springer, Heidelberg, 2012. \MR{2934577}

\bibitem[MV99]{MV}
Fabien Morel and Vladimir Voevodsky, \emph{{${\bf A}^1$}-homotopy theory of
  schemes}, Inst. Hautes \'{E}tudes Sci. Publ. Math. (1999), no.~90, 45--143
  (2001). \MR{1813224}

\bibitem[SS75]{SchejaStorch}
G\"{u}nter Scheja and Uwe Storch, \emph{\"{U}ber spurfunktionen bei
  vollst\"{a}ndigen durchschnitten.}, Journal f\"{u}r die reine und angewandte
  Mathematik \textbf{0278\_0279} (1975), 174--190.

\bibitem[SW18]{SW}
Padmavathi Srinivasan and Kirsten Wickelgren, \emph{An arithmetic count of the
  lines meeting four lines in $\mathbf{P}^{3}$}, arXiv preprint
  arXiv:1810.03503 (2018).

\bibitem[WW19]{unstablehandbook}
Kirsten Wickelgren and Ben Williams, \emph{Unstable motivic homotopy theory},
  2019.

\end{thebibliography}
\end{document}